\pdfoutput=1
\documentclass[11pt, reqno]{amsart}

\usepackage[utf8]{inputenc}
\usepackage[T1]{fontenc}
\usepackage[draft=false,kerning=true]{microtype}
\usepackage[a4paper, left=2.5cm, right=2.5cm, top=3.5cm, bottom=3cm]{geometry}

\usepackage[charter]{mathdesign}

\usepackage{graphicx}
\usepackage{tikz}
\usetikzlibrary{arrows,backgrounds,calc}
\usepackage{pgfplots}
\pgfplotsset{compat=1.13}

\usepackage{amsthm, amsmath, amsfonts}
\usepackage{mathtools}
\usepackage{enumerate}
\usepackage{subcaption}
\usepackage{todonotes}

\usepackage{hyperref}

\newtheorem{theorem}{Theorem}
\newtheorem{proposition}{Proposition}[section]

\theoremstyle{remark}
\newtheorem*{remark}{Remark}

\theoremstyle{definition}
\newtheorem*{definition}{Definition}

\renewcommand{\MR}[1]{}

\DeclareSymbolFont{cmcal}{OMS}{cmsy}{m}{n}
\SetSymbolFont{cmcal}{bold}{OMS}{cmsy}{b}{n}
\DeclareSymbolFontAlphabet{\mathcal}{cmcal}

\newcommand{\R}[0]{\mathbb{R}}
\newcommand{\Z}[0]{\mathbb{Z}}
\renewcommand{\P}[0]{\mathbb{P}}
\newcommand{\E}[0]{\mathbb{E}}
\newcommand{\V}[0]{\mathbb{V}}

\newcommand{\tree}[0]{T}

\DeclarePairedDelimiter{\abs}{\lvert}{\rvert}
\DeclarePairedDelimiter{\iverson}{\llbracket}{\rrbracket}

\title[Iterative Leaf Cutting]{Reducing Simply Generated Trees by Iterative Leaf Cutting}

\author[B.~Hackl]{Benjamin Hackl}
\author[C.~Heuberger]{Clemens Heuberger}

\address[Benjamin Hackl, Clemens Heuberger]
{Institut f\"ur Mathematik,
  Alpen-Adria-Uni\-ver\-si\-t\"at Klagenfurt, Universit\"atsstra\ss e
  65--67, 9020 Klagenfurt, Austria}
\email{\href{mailto:benjamin.hackl@aau.at}{benjamin.hackl@aau.at}}
\email{\href{mailto:clemens.heuberger@aau.at}{clemens.heuberger@aau.at}}
\thanks{B.~Hackl and C.~Heuberger are supported by the Austrian
  Science Fund (FWF): P~28466-N35.}

\author[S.~Wagner]{Stephan Wagner}
\address[Stephan Wagner]{
  Department of Mathematical Sciences, Stellenbosch University,
  7602 Stellenbosch, South Africa}
\email{\href{mailto:swagner@sun.ac.za}{swagner@sun.ac.za}}
\thanks{S.~Wagner is supported by the National Research Foundation of South Africa, grant 96236.}


\begin{document}

\begin{abstract}
  We consider a procedure to reduce simply generated trees by iteratively removing
  all leaves. In the context of this reduction, we study the number of vertices that are
  deleted after applying this procedure a fixed number of times by using an additive tree
  parameter model combined with a recursive characterization.

  Our results include asymptotic formulas for mean and variance of this quantity as well
  as a central limit theorem.
\end{abstract}
\maketitle

\section{Introduction}\label{sec:introduction}

Trees are one of the most fundamental combinatorial structures with a plethora of
applications not only in mathematics, but also in, e.g., computer science or biology. A
matter of recent interest in the study of trees is the question of how a given tree family behaves
when applying a fixed number of iterations of some given deterministic reduction procedure
to it. See~\cite{Hackl-Heuberger-Kropf-Prodinger:ta:treereductions,
  Hackl-Prodinger:ta:catalan-stanley} for the study of different reduction procedures on
(classes of) plane trees, and \cite{Hackl-Heuberger-Prodinger:2018:register-reduction} for
a reduction procedure acting on binary trees related to the register function.

In the scope of this extended abstract we focus on the, in a sense, most natural reduction
procedure: we reduce a given rooted tree by cutting off all leaves so that only internal
nodes remain. This process is illustrated in Figure~\ref{fig:leaf-reduction}. While in
this extended abstract we are mainly interested in the family of simply generated trees,
further families of rooted trees will be investigated in the full version.

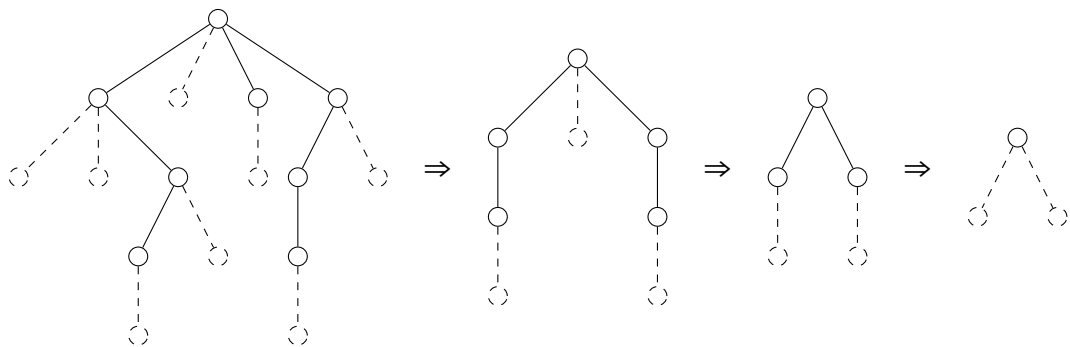
\begin{figure}[ht]
  \centering
  \begin{tikzpicture}[scale=0.7, baseline=(current bounding box.center), inner sep=2.5pt]
    \node[draw, circle] {}
    child {node[draw, circle] {}
      child[dashed] {node[draw, circle] {}}
      child[dashed] {node[draw, circle] {}}
      child {node[draw, circle] {}
        child {node[draw, circle] {}
          child[dashed] {node[draw, circle] {}}
        }
        child[dashed] {node[draw, circle] {}}
      }
    }
    child[dashed] {node[draw, circle] {}}
    child {node[draw, circle] {}
      child[dashed] {node[draw, circle] {}}
    }
    child {node[draw, circle] {}
      child {node[draw, circle] {}
        child {node[draw, circle] {}
          child[dashed] {node[draw, circle] {}}
        }
      }
      child[dashed] {node[draw, circle] {}}
    };
  \end{tikzpicture} $\quad\Rightarrow\quad$
  \begin{tikzpicture}[scale=0.7, baseline=(current bounding box.center), inner sep=2.5pt]
    \node[draw, circle] {}
    child {node[draw, circle] {}
      child {node[draw, circle] {}
        child[dashed] {node[draw, circle] {}}
      }
    }
    child[dashed] {node[draw, circle] {}}
    child {node[draw, circle] {}
      child {node[draw, circle] {}
        child[dashed] {node[draw, circle] {}}
      }
    };
  \end{tikzpicture} $\quad\Rightarrow\quad$
  \begin{tikzpicture}[scale=0.7, baseline=(current bounding box.center), inner sep=2.5pt]
    \node[draw, circle] {}
    child {node[draw, circle] {}
      child[dashed] {node[draw, circle] {}}
    }
    child {node[draw, circle] {}
      child[dashed] {node[draw, circle] {}}
    };
  \end{tikzpicture} $\quad\Rightarrow\quad$
  \begin{tikzpicture}[scale=0.7, baseline=(current bounding box.center), inner sep=2.5pt]
    \node[draw, circle] {}
    child[dashed] {node[draw, circle] {}}
    child[dashed] {node[draw, circle] {}};
  \end{tikzpicture}
  \caption{Multiple applications of the ``cutting leaves'' process to
    a given rooted tree}
  \label{fig:leaf-reduction}
\end{figure}

It is easy to see that the number of steps it takes to reduce the tree so
that only the root remains is precisely the height of the tree, i.e., the greatest
distance from the root to a leaf. A more delicate question---the one in the center of this
article---is to ask for a precise analysis of the number of vertices deleted when
applying the ``cutting leaves'' reduction a fixed number of times.

The key concepts behind our analysis are a recursive characterization and bivariate
generating functions. Details on our model are given in Section~\ref{sec:tree-parameter}.
The asymptotic analysis is then carried out in Section~\ref{sec:simply-generated}, with our
main result given in Theorem~\ref{thm:simply-generated}. It includes precise asymptotic
formulas for the mean and variance of the number of removed vertices when applying the
reduction a fixed number of times. Furthermore, we also prove a central limit theorem.

Finally, in Section~\ref{sec:outlook} we give an outlook on the analysis of the ``cutting
leaves'' reduction in the context of other classes of rooted trees. Qualitative results
for these classes are given in Theorem~\ref{thm:outlook:qualitative}. The corresponding
details will be published in the full version of this extended abstract.

The computational aspects in this extended abstract were carried out using the module
for manipulating asymptotic expansions~\cite{Hackl-Heuberger-Krenn:2016:asy-sagemath} 
in the free open-source mathematics software system SageMath~\cite{SageMath:2018:8.2}.
A notebook containing our calculations can be found at
\begin{center}
\url{https://benjamin-hackl.at/publications/iterative-leaf-cutting/}.
\end{center}

\section{Preliminaries}\label{sec:tree-parameter}

So-called \emph{additive tree parameters} play an integral part in our analysis of the
number of removed nodes.

\begin{definition}
A \emph{fringe subtree} of a rooted tree is a subtree that consists of a vertex and all its descendants.
  An \emph{additive tree parameter} is a functional $F$ satisfying a recursion of the form
  \begin{equation}\label{eq:additive-tree-parameter}
    F(\tree) = \sum_{j=1}^{k} F(\tree^{j}) + f(\tree),
  \end{equation}
  where $\tree$ is some rooted tree, $\tree^{1}$, $\tree^{2}$, \ldots, $\tree^{k}$ are the branches of the root
  of $\tree$, i.e., the fringe subtrees rooted at the children of the root of $\tree$, and $f$ is
  a so-called \emph{toll function}.
\end{definition}
There are several recent articles on properties of additive tree parameters, see for
example \cite{Wagner:2015:centr-limit}, \cite{Janson:2016:normality-add-func},
and~\cite{Ralaivaosaona-Wagner:2018:d-ary-increasing}.

It is easy to see that such an additive tree parameter can be computed by
summing the toll function over all fringe subtrees, i.e., if $\tree^{(v)}$ denotes
the fringe subtree rooted at the vertex $v$ of $\tree$, then we have
\[ F(\tree) = \sum_{v\in \tree} f(\tree^{(v)}).  \]
In particular, the parameter is fully determined by specifying the toll function $f$.

Tree parameters play an important role in our analysis because our quantity of
interest---the number of removed vertices when applying the ``cutting leaves'' reduction
$r$ times---can be seen as such a parameter. Let $a_r(\tree)$ denote this parameter for a given rooted
tree $\tree$.

\begin{proposition}\label{prop:toll-function}
  The toll function belonging to $a_r(\tree)$ is given by
  \begin{equation}\label{eq:toll-function}
    f_r(\tree) = \begin{cases}
      1 & \text{ if the height of } \tree \text{ is less than } r,\\
      0 & \text { else.}
    \end{cases}
  \end{equation}
  In other words, if $\mathcal{T}_r$ denotes the family of rooted trees of height less
  than $r$, the toll function can be written in Iverson notation\footnote{The Iverson
    notation, as popularized in~\cite{Graham-Knuth-Patashnik:1994}, is defined as follows:
    $\iverson{\mathrm{expr}}$ evaluates to 1 if $\mathrm{expr}$ is true, and to 0
    otherwise.} as $f_r(\tree) = \iverson{\tree \in \mathcal{T}_r}$.
\end{proposition}
\begin{proof}
  It is easy to see that the number of removed vertices satisfies this additive
  property---the number of deleted nodes in some tree $\tree$ is precisely the sum
  of all deleted nodes in the branches of $\tree$ in case the root is not deleted. Otherwise,
  the sum has to be increased by one to account for the root node. Thus, the toll function
  determines whether or not the root node of $\tree$ is deleted.

  The fact that the root node is deleted if and only if the number of reductions $r$ is
  greater than the height of the tree is already illustrated in Figure~\ref{fig:leaf-reduction}.
\end{proof}

Basically, our strategy to analyze the quantity $a_r(\tree)$ for simply generated families
of trees uses the recursive structure of~\eqref{eq:additive-tree-parameter} together
with the structure of the family itself to derive a functional equation for a suitable
bivariate generating function $A_r(x,u)$. In this context, the trees $\tree$ in the family
$\mathcal{T}$ are enumerated with respect to their size (corresponding to the variable
$x$) and the value of the parameter $a_r(\tree)$ (corresponding to the variable $u$).

Throughout the remainder of this extended abstract, $\mathcal{T}$ denotes the family of trees under
investigation, and for all $r\in\Z_{\geq 1}$, $\mathcal{T}_r\subset\mathcal{T}$ denotes
the class of trees of height less than $r$. The corresponding generating functions are
denoted by $F(x)$ and $F_r(x)$.

Furthermore, from now on, $\tree_n$ denotes a random\footnote{The underlying probability
  distribution will always be clear from context.} tree of size $n$ (i.e., a tree that
consists of $n$ vertices) from $\mathcal{T}$. This means that formally, the quantity we
are interested in analyzing is the random variable $a_r(\tree_n)$ for large $n$.

\section{Reducing Simply Generated Trees}\label{sec:simply-generated}

\subsection{Recursive Characterization}\label{sec:simply-generated:characterization}

Let us begin by recalling the definition of simply generated trees. A simply generated
family of trees $\mathcal{T}$ can be defined by imposing a weight function on plane
trees. For a sequence of nonnegative weights $(w_k)_{k\geq 0}$ (we will make the customary
assumption that $w_0 = 1$ without loss of generality; cf.~\cite[Section
4]{Janson:2012:simply-generated-survey}), one defines the weight of a rooted ordered tree
$\tree$ as the product
\[ w(\tree) \coloneqq \prod_{j \geq 0} w_j^{N_j(\tree)}, \]
where $N_j(\tree)$ is the number of vertices in $\tree$ with precisely $j$ children. The weight
generating function
\begin{equation}\label{eq:f_def}
  F(x) = \sum_{\tree\in\mathcal{T}} w(\tree) x^{|\tree|},
\end{equation}
where $\abs{\tree}$ denotes the size of $\tree$ and where the sum is over all plane trees,
is easily seen to satisfy a functional equation. By setting
$\Phi(t) = \sum_{j \geq 0} w_j t^j$ and applying the symbolic method
(see~\cite[Chapter~I]{Flajolet-Sedgewick:ta:analy}) to decompose a simply generated
tree as the root node with some simply generated trees attached, we have
\begin{equation}\label{eq:sg_trees}
  F(x) = x \Phi(F(x)).
\end{equation}
We define a probability measure on the set of all rooted ordered trees with
$n$ vertices by assigning a probability proportional to $w(\tree)$ to every tree $\tree$.

\begin{remark}
  Several important families of trees are covered by suitable choices of weights:
  \begin{itemize}
  \item[--] plane trees are obtained from the weight sequence with $w_j = 1$ for all
    $j$,
  \item[--] labelled trees correspond to weights given by $w_j = \frac{1}{j!}$,
  \item[--] and $d$-ary trees
  (where every vertex has either $d$ or no children) are obtained by setting $w_0 = w_d = 1$
  and $w_j = 0$ for all other $j$.
  \end{itemize}
\end{remark}

In the context of simply generated trees, it is natural to define the bivariate generating
function $A_r(x,u)$ to be a weight generating function, i.e.,
\[ A_r(x,u) = \sum_{\tree\in\mathcal{T}} w(\tree) x^{|\tree|} u^{a_r(\tree)}. \]

As explicitly stated in Proposition~\ref{prop:toll-function}, the combinatorial class
$\mathcal{T}_r$ of trees of height less than $r$ is integral for deriving a functional
equation for $A_r(x,u)$. Write $F_r(x)$ for the weight generating function associated with
$\mathcal{T}_r$, defined in the same way as $F(x)$ in~\eqref{eq:f_def}.

Clearly, $F_1(x) = x$, since there is only one rooted tree of height $0$, which only
consists of the root. Moreover, via the decomposition mentioned in the interpretation
of~\eqref{eq:sg_trees}, we have
\begin{equation}\label{eq:iteration}
  F_r(x) = x \Phi(F_{r-1}(x))
\end{equation} for every $r > 1$.

Now we are prepared to derive the aforementioned functional equation.

\begin{proposition}\label{prop:simply-generated:functional-equation}
  The bivariate weight generating function $A_r(x,u)$ satisfies the functional equation
  \begin{equation}\label{eq:simply-generated:functional-equation}
    A_r(x,u) = x\Phi(A_r(x,u)) + \Big( 1 - \frac{1}{u} \Big) F_r(xu).
  \end{equation}
\end{proposition}
\begin{proof}
  We can express the
sum over all trees $\tree$ in the definition of $A_r(x,u)$ as a sum over all possible root
degrees $k$ and $k$-tuples of branches. In view of~\eqref{eq:toll-function}, this gives us
\begin{align*}
  A_r(x,u) & = \sum_{\tree\in\mathcal{T}\setminus\mathcal{T}_{r}} w(\tree) x^{\abs{\tree}}u^{a_{r}(\tree)}
             + \sum_{\tree\in\mathcal{T}_{r}} w(\tree) x^{\abs{\tree}} u^{\abs{\tree}}\\
           & = \sum_{k \geq 0} w_k \sum_{\tree^{1}\in\mathcal{T}} \cdots
             \sum_{\tree^{k}\in\mathcal{T}} \Big(\prod_{j=1}^k w(\tree^{j}) \Big) x^{1+|\tree^{1}|+\cdots+|\tree^{k}|}
             u^{a_r(\tree^{1})+\cdots+a_r(\tree^{k})}\\
           & \quad + \sum_{\tree \in \mathcal{T}_r} w(\tree) x^{|\tree|} \big(u^{|\tree|} - u^{|\tree|-1}\big) \\
           & = x \sum_{k \geq 0} w_k \Big( \sum_{\tree\in\mathcal{T}} w(\tree) x^{|\tree|} u^{a_r(\tree)} \Big)^k +
             \Big(1 - \frac{1}{u} \Big) \sum_{\tree \in \mathcal{T}_r} w(\tree) (xu)^{|\tree|} \\
           & = x\Phi(A_r(x,u)) + \Big( 1 - \frac{1}{u} \Big) F_r(xu). \qedhere
\end{align*}
\end{proof}
\begin{remark}
  Setting $u=1$ reduces this functional equation to~\eqref{eq:sg_trees}, with
  $A_r(x,1) = F(x)$.
\end{remark}

The functional equation~\eqref{eq:simply-generated:functional-equation} provides enough
leverage to carry out a full asymptotic analysis of the behavior of $a_r(\tree_{n})$ for simply
generated trees.

\subsection{Parameter Analysis}\label{sec:simply-generated:analysis}

Now we use the functional equation to determine mean and variance of
$a_r$, which are obtained from the partial derivatives with respect to $u$, evaluated at
$u=1$. To be more precise, if $\tree_n$ denotes a random (with respect to the probability
distribution determined by the given weight sequence) simply generated tree of size $n$,
then after normalization, the factorial moments
\[ \E a_r(\tree_n)^{\underline{k}} \coloneqq \E (a_r(\tree_n) (a_r(\tree_n) - 1) \cdots (a_r(\tree_n)
  - k + 1)) \]
can be extracted as the coefficient of $x^n$ in the partial derivative
$\frac{\partial^k\,}{\partial u^k} A_r(x,u)\big|_{u=1}$. And from there, expectation and
variance can be obtained in a straightforward way.

From this point on, we make some reasonable assumptions on the weight sequence
$(w_k)_{k\geq 0}$. In addition to $w_0 = 1$, we assume that there is a $k > 1$ with $w_k >
0$ to avoid trivial cases. Furthermore, we require that if $R > 0$ is the radius of
convergence of the weight generating function $\Phi(t) = \sum_{k\geq 0} w_k t^k$,
there is a unique positive $\tau$ (the \emph{fundamental constant}) with $0 < \tau < R$
such that $\Phi(\tau) - \tau \Phi'(\tau) = 0$. This is to ensure that the singular
behavior of $F(x)$ can be fully characterized (see, e.g.,
\cite[Section~VI.7]{Flajolet-Sedgewick:ta:analy}).

\begin{proposition}
  Let $r \in \Z_{\geq 1}$ be fixed, let $\mathcal{T}$ be a simply generated family of
  trees and let $\mathcal{T}_r \subset \mathcal{T}$ be the set of trees with height less
  than $r$. If $\tree_n$ denotes a random tree from $\mathcal{T}$ of size $n$ (with respect
  to the probability measure defined on $\mathcal{T}$), then for $n\to\infty$ the expected
  number of removed nodes when applying the ``cutting leaves'' procedure $r$ times to
  $\tree_n$ and the corresponding variance satisfy
  \begin{equation}\label{eq:simply-generated:expectation-variance}
    \E a_r(\tree_n) = \mu_r n + \frac{\rho \tau^2 F_r'(\rho) +
      3\beta\tau F_r(\rho) - \alpha^2 F_r(\rho)}{2\tau^3} + O(n^{-1}),
    \quad\text{ and }\quad
    \V a_r(\tree_n) = \sigma_r^2 n + O(1).
  \end{equation}
  The constants $\mu_r$ and $\sigma_r^2$ are given by
  \begin{equation}\label{eq:simply-generated:constants}
    \mu_r = \frac{F_r(\rho)}{\tau},\qquad \sigma_r^2 = \frac{4 \rho \tau^3 F_r'(\rho)
      - 4 \rho \tau^2 F_r(\rho)F_r'(\rho) + (2\tau^2 - \alpha^2) F_r(\rho)^2 - 2\tau^3 F_r(\rho)}{2\tau^4},
  \end{equation}
  where $F_r(x)$ is the weight generating function corresponding to $\mathcal{T}_r$,
  $\rho$ is the radius of convergence of $F(x)$ and given by $\rho = \tau/\Phi(\tau)$, and
  the constants $\alpha$ and $\beta$ are given by
  \[ \alpha = \sqrt{\frac{2\tau}{\rho \Phi''(\tau)}}, \qquad
    \beta = \frac{1}{\rho\Phi''(\tau)} -
    \frac{\tau \Phi'''(\tau)}{3\rho \Phi''(\tau)^2}. \]
\end{proposition}
\begin{remark}
  For the sake of technical convenience, we are going to assume that $\Phi(t)$ is an
  aperiodic power series, meaning that the period $p$, i.e., the greatest common divisor
  of all indices $j$ for  which $w_j\neq 0$, is 1. This implies (see 
  \cite[Theorem~VI.6]{Flajolet-Sedgewick:ta:analy}) that $F(x)$ has a unique square root
  singularity located at $\rho = \tau/\Phi(\tau)$, which makes some of our computations less
  tedious. However, all of our results also apply (mutatis mutandis) if this aperiodicity
  condition is not satisfied---with the restriction that then, $n-1$ has to be a multiple of
  the period $p$.
\end{remark}
\begin{proof}
First, we have
\[
  \frac{\partial}{\partial u} A_r(x,u) =
  x \Phi'(A_r(x,u)) \frac{\partial}{\partial u} A_r(x,u)
  + \frac{1}{u^2} F_r(xu) + x \Big( 1 - \frac{1}{u} \Big) F_r'(xu),
\]
so
\[
  \frac{\partial}{\partial u} A_r(x,u) \Big|_{u=1} =
  \frac{F_r(x)}{1-x \Phi'(A_r(x,1))} = \frac{F_r(x)}{1-x \Phi'(F(x))}.
\]
Analogously, we can use implicit differentiation on~\eqref{eq:sg_trees} to obtain
\[
  x F'(x) = \frac{F(x)}{1-x \Phi'(F(x))},
\]
so
\[
  \frac{\partial}{\partial u} A_r(x,u) \Big|_{u=1} =
  \frac{x F'(x) F_r(x)}{F(x)}.
\]
The second derivative is found in the same way: we obtain
\[
  \frac{\partial^2}{\partial u^2} A_r(x,u) \Big|_{u=1} =
  \frac{2(x F_r'(x) - F_r(x))}{1-x \Phi'(F(x))}
  + \frac{x F_r(x)^2\Phi''(F(x))}{(1-x \Phi'(F(x)))^3}
\]
by differentiating implicitly a second time. Again, this can be
expressed in terms of the derivatives of~$F$:
\[
  \frac{\partial^2}{\partial u^2} A_r(x,u) \Big|_{u=1} =
  \Big( \frac{2F_r(x)^2}{F(x)^2} - \frac{2F_r(x)}{F(x)}
  + \frac{2x F_r'(x)}{F(x)} \Big) x F'(x)
  + \frac{x^2 F_r(x)^2 F''(x)}{F(x)^2} - \frac{2x^2 F_r(x)^2 F'(x)^2}{F(x)^3}.
\]

By the assumptions made in this section, there is a positive real number $\tau$ that is
smaller than the radius of convergence of $\Phi$ and satisfies the equation
$\tau \Phi'(\tau) = \Phi(\tau)$. It is well known
(see~\cite[Section~VI.7]{Flajolet-Sedgewick:ta:analy}) that in this case, $F(x)$ has a
square root singularity at $\rho = \tau/\Phi(\tau) = 1/\Phi'(\tau)$, with singular
expansion
\begin{equation}\label{eq:simply-generated:F-expansion}
  F(x) = \tau - \alpha \sqrt{1-x/\rho} + \beta (1-x/\rho) + O((1-x/\rho)^{3/2}).
\end{equation}
Here, the coefficients $\alpha$ and $\beta$ are given by
\begin{equation}\label{eq:alpha}
  \alpha = \sqrt{\frac{2\tau}{\rho \Phi''(\tau)}}
\end{equation}
and
\[
  \beta = \frac{1}{\rho \Phi''(\tau)}
  - \frac{\tau \Phi'''(\tau)}{3\rho \Phi''(\tau)^2}
\]
respectively. Note that in case more precise asymptotics are desired, further terms of
the singular expansion can be computed easily.

Due to our aperiodicity assumption, $\rho$ is the only singularity on $F$'s circle of
convergence, and the conditions of singularity analysis (see
\cite{Flajolet-Odlyzko:1990:singul} or \cite[Chapter~VI]{Flajolet-Sedgewick:ta:analy}, for
example) are satisfied.

Next we note that $F_r$ has greater radius of convergence than $F$. This follows
from~\eqref{eq:iteration} by induction on $r$: it is clear for $r=1$, and if $F_{r-1}$ is
analytic at $\rho$, then so is $F_r$, since $|F_{r-1}(\rho)| < F(\rho) = \tau$ is smaller
than the radius of convergence of $\Phi$. So $F_r$ has greater radius of convergence than
$F$. This implies that $F_r$ has a Taylor expansion around $\rho$:
\[
  F_r(x) = F_r(\rho) - \rho F_r'(\rho) (1 - x/\rho) + O((1-x/\rho)^2).
\]
We find that
\begin{align*}
  \frac{\partial}{\partial u} A_r(x,u) \Big|_{u=1} = \frac{x F'(x) F_r(x)}{F(x)}
  & = \frac{F_r(\rho)}{\tau} \cdot (x F'(x))
    + \frac{\rho \tau^2 F_r'(\rho) + 3\beta \tau F_r(\rho) - \alpha^2 F_r(\rho)}{2\tau^3}
    \cdot F(x) \\
  & \quad + C_1 + C_2 (1-x/\rho) + O((1-x/\rho)^{3/2}).
\end{align*}
for certain constants $C_1$ and $C_2$.

The $n$th coefficient of the derivative $[z^n]\frac{\partial}{\partial u}
A_r(x,u) \big|_{u=1}$ can now be extracted by means of singularity analysis. Normalizing
the result by dividing by $[z^n] A_r(x,1) = [z^n] F(x)$ (again extracted by means of
singularity analysis; the corresponding expansion is given
in~\eqref{eq:simply-generated:F-expansion}) yields an asymptotic expansion for $\E
a_r(\tree_n)$. We find
\begin{equation}\label{eq:mean_sg}
  \E a_r(\tree_n) = \frac{F_r(\rho)}{\tau} \cdot n
  + \frac{\rho \tau^2 F_r'(\rho) + 3\beta \tau F_r(\rho) - \alpha^2 F_r(\rho)}{2\tau^3}
  + O(n^{-1}).
\end{equation}
Similarly, from
\begin{align*}
  \frac{\partial^2}{\partial u^2} A_r(x,u) \Big|_{u=1}
  &= \Big(\frac{2F_r(x)^2}{F(x)^2} - \frac{2F_r(x)}{F(x)}
    + \frac{2x F_r'(x)}{F(x)} \Big) x F'(x)
    + \frac{x^2 F_r(x)^2 F''(x)}{F(x)^2} - \frac{2x^2 F_r(x)^2 F'(x)^2}{F(x)^3} \\
  &= \Big(\frac{F_r(\rho)}{\tau} \Big)^2 \cdot (x^2F''(x)+xF'(x)) \\
  &\quad + \frac{4\rho \tau^3 F_r'(\rho) - 2\rho \tau^2 F_r(\rho) F_r'(\rho)
    + (2\tau^2 + 6\beta \tau - 3\alpha^2)F_r(\rho)^2
    - 4\tau^3 F_r(\rho)}{2\tau^4} \cdot (xF'(x)) \\
  &\quad + C_3 + O((1-x/\rho)^{1/2}),
\end{align*}
we can use singularity analysis to find an asymptotic expansion for the second factorial
moment $\E a_r(\tree_n)^{\underline{2}}$. Plugging the result and the expansion for the
mean from~(\ref{eq:mean_sg}) into the well-known identity
\[ \V a_r(\tree_n) = \E a_r(\tree_n)^{\underline{2}} + \E a_r(\tree_n) - (\E a_r(\tree_n))^2 \]
then yields
\[
  \V a_r(\tree_n) =
  \frac{4 \rho \tau^3 F_r'(\rho) - 4 \rho \tau^2 F_r(\rho)F_r'(\rho)
    + (2\tau^2 - \alpha^2) F_r(\rho)^2 - 2\tau^3 F_r(\rho)}{2\tau^4} \cdot n + O(1).
  \qedhere
\]
\end{proof}

While this analysis provided us with a precise characterization for the mean and the
variance of the number of deleted vertices, it would be interesting to have more
information on how these quantities behave for a very large number of iterated
reductions. The following proposition gives more details on the main contribution.

\begin{proposition}\label{prop:simply-generated:constants-asy}
  For $r\to\infty$, the constants $\mu_r$ and $\sigma_r^2$ admit the asymptotic expansions
  \begin{equation}\label{eq:simply-generated:constants-asy}
    \mu_r = 1 - \frac{2}{\rho\tau\Phi''(\tau)} r^{-1} + o(r^{-1}) \quad\text{ and }\quad
    \sigma_r^2 = \frac{1}{3\rho\tau\Phi''(\tau)} + o(1),
  \end{equation}
  respectively.
\end{proposition}
\begin{proof}
  In order to obtain the behavior of $\mu_r$ and $\sigma_r^2$ for $r\to\infty$, we have to
  study the behavior of $c_r = F_r(\rho)$ and $d_r = F_r'(\rho)$ as $r \to \infty$. First,
  we have the recursion
  \[ c_r = \rho \Phi(c_{r-1}). \]
  We note that $c_r$ is increasing in $r$ (since the
  coefficients of $F_r(x)$ are all nondecreasing in $r$ in view of the combinatorial
  interpretation), and $c_r \to \tau$ as $r \to \infty$. By Taylor expansion around
  $\tau$, we obtain
  \[
    \tau - c_r = \rho \Phi(\tau) - \rho \Phi(c_{r-1})
    = \rho \Phi'(\tau) (\tau - c_{r-1})
    - \frac{\rho \Phi''(\tau)}{2} (\tau - c_{r-1})^2
    + O((\tau - c_{r-1})^3),
  \]
  and since $\rho \Phi'(\tau) = 1$, it follows that
  \[
    \frac{1}{\tau - c_r} = \frac{1}{\tau - c_{r-1}}
    + \frac{\rho \Phi''(\tau)}{2} + O(\tau - c_{r-1}).
  \]
  Now we can conclude that
  \[
    \frac{1}{\tau - c_r} = \frac{\rho \Phi''(\tau)}{2} r + o(r),
  \]
  so
  \[
    \mu_r = \frac{c_r}{\tau}
    = 1 - \frac{2}{\rho \tau \Phi''(\tau)} r^{-1} + o(r^{-1}).
  \]
  Further terms can be derived by means of bootstrapping. Similarly, differentiating the
  identity $F_r(x) = x \Phi(F_{r-1}(x))$ gives us the recursion
  \[ d_r = \rho \Phi'(c_{r-1}) d_{r-1} + \frac{c_r}{\rho}. \]
  The sequence $d_r$ is increasing for the same reason $c_r$ is. Moreover, since
  $\rho \Phi'(c_{r-1}) < \rho \Phi'(\tau) = 1$, it follows from the recursion that
  $d_r = O(r)$. Now, we use Taylor expansion again to obtain
  \begin{align*}
    d_r &= \rho \big(\Phi'(\tau) - \Phi''(\tau) (\tau-c_{r-1})
          + O((\tau - c_{r-1})^2)\big) d_{r-1}
          + \frac{c_r}{\rho} \\
        &= \big(\rho \Phi'(\tau) - \rho\Phi''(\tau) (\tau-c_{r-1})
          + O((\tau - c_{r-1})^2) \big) d_{r-1}
          + \frac{c_r}{\rho} \\
        &= \Big( 1 - \frac{2}{r} + o(r^{-1}) \Big) d_{r-1}
          + \frac{\tau}{\rho} + o(1).
  \end{align*}
  This can be rewritten as $r^2 d_r = (r-1)^2 d_{r-1} + \frac{\tau}{\rho} r^2 + o(r^2)$,
  which gives us $r^2 d_r = \frac{\tau}{3\rho} r^3 + o(r^3)$ and allows us to conclude that
  \[ d_r = \frac{\tau}{3\rho} r + o(r). \]
  Plugging the formulas for $c_r$ and $d_r$ into~\eqref{eq:simply-generated:constants}, we
  find that
  \[
    \sigma_r^2 = \frac{1}{3\rho \tau \Phi''(\tau)} + o(1).
    \qedhere
  \]
\end{proof}

As a side effect of Proposition~\ref{prop:simply-generated:constants-asy}, we can also
observe that for sufficiently large $r$, the constant $\sigma_r^2$ is strictly
positive. As a consequence, the parameter $a_r(\tree_n)$ is asymptotically normally
distributed in these cases.

However, we can do even better: we can prove that $a_{r}(\tree_{n})$ always admits a
Gaussian limit law, except for an---in some sense---pathological case.

\begin{proposition}\label{prop:simply-generated:limit-law}
  Let $\mathcal{T}$ be a simply generated family of trees and fix $r\in\Z_{\geq 1}$. Then
  the random variable $a_{r}(\tree_{n})$ is asymptotically normally distributed, except in
  the case of $d$-ary trees when $r=1$. In all other cases we find that for $x\in\R$ we have
  \[ \P\Big(\frac{a_{r}(\tree_{n}) - \mu_{r}n}{\sqrt{\sigma_{r}^{2} n}} \leq x\Big) =
    \frac{1}{\sqrt{2\pi}} \int_{-\infty}^{x} e^{-t^{2}/2}~dt + O(n^{-1/2}). \]
\end{proposition}
\begin{proof}
  Observe that as soon as we are able to prove that the variance of $a_{r}(\tree_{n})$ is
  actually linear with respect to $n$, i.e., $\sigma_{r}^{2} \neq 0$, all conditions
  of~\cite[Theorem~2.23]{Drmota:2009:random} hold and are checked easily, thus proving a
  Gaussian limit law.

  Our strategy for proving a linear lower bound for the variance relies on choosing two
  trees $\tree^{1}$, $\tree^{2}\in\mathcal{T}$ with $\abs{\tree^{1}} = \abs{\tree^{2}}$
  such that $a_{r}(\tree^{1}) \neq a_{r}(\tree^{2})$, and $a_{r}(\tree^{1}), a_{r}(\tree^{2}) < \abs{\tree^{1}}$
(i.e., neither of the two is completely reduced after $r$ steps, and the number of vertices removed
after $r$ steps differs between $\tree^{1}$ and  $\tree^{2}$).
  While this is not possible in the
  case where $r=1$ and $\mathcal{T}$ is a family of $d$-ary trees (where the number of
  leaves, and thus the number of removed nodes when cutting the tree once only depends 
  on the tree size), such trees can always be found in all other cases. To be
  more precise, if $r = 1$ and $\mathcal{T}$ is not a $d$-ary family of trees, there have
  to be at least three different possibilities for the number of children, namely $0$, $d$, and $e$.
  Then, in a sufficiently large tree $\tree^{1}$, the number of inner nodes with $d$
  children can be reduced by $e$, and the number of inner nodes with $e$ children can be
  increased by $d$ in order to obtain a tree $\tree^{2}$ of the same size with a different
  number of leaves, thus satisfying our conditions.

  For the case $r\geq 2$, observe that the problem above cannot arise as cutting the leaves
  off some tree in $\mathcal{T}$ does not necessarily yield another tree in $\mathcal{T}$.
  Let $d$ be a positive integer for which the weight $w_{d}$ is positive (i.e., a node in
  a tree in $\mathcal{T}$ can have $d$ children). We choose $\tree^{1}$ to be the
  complete $d$-ary tree of height $r$. A second $d$-ary tree $\tree^{2}$ is then
  constructed by arranging the same number of internal vertices as a path and by attaching
  suitably many leaves. The handshaking lemma then guarantees that both trees have the
  same size; but $a_{r}(\tree^{1})$ is obviously larger than $a_{r}(\tree^{2})$.

  It is well-known (see e.g.~\cite{Aldous:1991:asy-fringe-distributions}
  or~\cite{Janson:2016:normality-add-func} for stronger results) that large trees (except
  for a negligible proportion) contain a linear (with respect to the size of the tree) number
  of copies of $\tree^{1}$ and $\tree^{2}$ as fringe subtrees. To be more precise, this
  means that there is a positive constant $c > 0$ such that the probability that a tree of
  size $n$ contains at least $cn$ copies of the patterns $\tree^{1}$ and $\tree^{2}$ is
  greater than $1/2$.

  Now, consider a large random
  tree $\tree$ in $\mathcal{T}$ and replace all occurrences of $\tree^{1}$ and $\tree^{2}$ by
  marked vertices. If $m$ denotes the number of marked vertices in the corresponding tree,
  then $m$ is of linear size with respect to the tree size $n$, except for a negligible
  proportion of trees. Given that, after replacing the patterns, the remaining tree contains $m$
  marked nodes, the number of occurrences of $\tree^{1}$ in the original (random) tree
  follows a binomial distribution with size parameter $m$ and probability $p =
  w(\tree^{1})/(w(\tree^{1})+w(\tree^{2})) \in (0,1)$ that only depends on the weights of
  the patterns. If we let $c_1$ and $c_2$ be the number of occurrences of $\tree^{1}$ and
$\tree^{2}$ respectively, then we have
$$a_{r}(\tree) = c_1a_{r}(\tree^{1}) + c_2a_{r}(\tree^{2}) + A,$$
where $A$ only depends on the shape of the reduced tree with $\tree^{1}$ and $\tree^{2}$
replaced by marked vertices. Let $M$ denote the random variable modeling the number of marked nodes
in the reduced tree obtained from a tree $\tree$ of size $n$. Then, via the law of total variance we find
  \[
    \V a_{r}(\tree_{n}) \geq \E(\V(a_{r}(\tree_{n}) | M)) \geq p(1-p)\E M \geq
    p(1-p)\frac{c}{2} n.
  \]
  The last inequality can be justified via the law of total
  expectation combined with the fact that the number of marked nodes in a tree with
  replaced patterns is at least $cn$ with probability greater than $1/2$. This
  proves that the variance of $a_{r}(\tree_{n})$ has to be of linear order.

  Finally, in order to prove that the speed of convergence is
  $O(n^{-1/2})$, we replace the formulation of Hwang's Quasi-Power Theorem without
  quantification of the speed of convergence (cf.~\cite[Theorem~2.22]{Drmota:2009:random})
  in the proof of~\cite[Theorem~2.23]{Drmota:2009:random} with a quantified version
  (see~\cite{Hwang:1998} or~\cite{Heuberger-Kropf:2016:higher-dimen} for a generalization
  to higher dimensions).
\end{proof}

The following theorem summarizes the results of the asymptotic analysis in this section.

\begin{theorem}\label{thm:simply-generated}
  Let $r\in\Z_{\geq 1}$ be fixed and $\mathcal{T}$ be a simply generated family of trees
  with weight generating function $\Phi$ and fundamental constant $\tau$, and set $\rho =
  \tau/\Phi(\tau)$. If $\tree_n$ denotes a random tree from $\mathcal{T}$ of size $n$ (with
  respect to the probability measure defined on $\mathcal{T}$), then for $n\to\infty$ the
  expected number of removed nodes when applying the ``cutting leaves'' procedure $r$
  times to $\tree_n$ and the corresponding variance satisfy
  \begin{equation*}
    \E a_r(\tree_n) = \mu_r n + \frac{\rho \tau^2 F_r'(\rho) +
      3\beta\tau F_r(\rho) - \alpha^2 F_r(\rho)}{2\tau^3} + O(n^{-1}),
    \quad\text{ and }\quad
    \V a_r(\tree_n) = \sigma_r^2 n + O(1).
  \end{equation*}
  The constants $\mu_r$ and $\sigma_r^2$ are given by
  \begin{equation*}
    \mu_r = \frac{F_r(\rho)}{\tau},\qquad \sigma_r^2 = \frac{4 \rho \tau^3 F_r'(\rho)
      - 4 \rho \tau^2 F_r(\rho)F_r'(\rho) + (2\tau^2 - \alpha^2) F_r(\rho)^2 - 2\tau^3 F_r(\rho)}{2\tau^4},
  \end{equation*}
  with
  \[ \alpha = \sqrt{\frac{2\tau}{\rho \Phi''(\tau)}}, \qquad
    \beta = \frac{1}{\rho\Phi''(\tau)} -
    \frac{\tau \Phi'''(\tau)}{3\rho \Phi''(\tau)^2}. \]
  Furthermore, for $r\to\infty$ the constants $\mu_r$ and $\sigma_r^2$ behave like
  \begin{equation*}
    \mu_r = 1 - \frac{2}{\rho\tau\Phi''(\tau)} r^{-1} + o(r^{-1}) \quad\text{ and }\quad
    \sigma_r^2 = \frac{1}{3\rho\tau\Phi''(\tau)} + o(1).
  \end{equation*}
  Finally, if $r \geq 2$ or $\mathcal{T}$ is not a family of $d$-ary trees, then
  $a_r(\tree_n)$ is asymptotically normally distributed, meaning that for $x\in\R$ we have
  \[ \P\Big(\frac{a_{r}(\tree_{n}) - \mu_{r}n}{\sqrt{\sigma_{r}^{2} n}} \leq x\Big) =
    \frac{1}{\sqrt{2\pi}} \int_{-\infty}^{x} e^{-t^{2}/2}~dt + O(n^{-1/2}). \]
\end{theorem}

\section{Outlook}\label{sec:outlook}

Our approach for analyzing the ``cutting leaves'' reduction procedure on simply generated
families of trees can be adapted to work for other families of trees as well. In this
section, we describe two additional classes of rooted trees to which our approach is
applicable and give qualitative results. Details on the analysis for these classes as well
as quantitative results will be given in the full version of this extended abstract.

The two additional classes of trees are \emph{P\'olya trees} and \emph{noncrossing
  trees}. P\'olya trees are unlabeled rooted trees where the ordering of the children is
not relevant. Uncrossing trees, on the other hand, are special labeled trees that satisfy
two conditions:
\begin{itemize}
\item[--] the root node has label $1$,
\item[--] when arranging the vertices in a circle such that the labels are sequentially
  ordered, none of the edges of the tree are crossing.
\end{itemize}
Obviously, noncrossing trees have their name from the second property. Both classes of
trees, P\'olya trees as well as noncrossing trees, are illustrated in
Figure~\ref{fig:more-tree-classes}.

\begin{figure}[ht]
  \centering
  \begin{subfigure}[b]{0.6\linewidth}
    \centering
    \begin{tikzpicture}[scale=0.7, inner sep=2.5pt]
      \node[draw, circle] {}
      child {node[draw, circle] {}
        child {node[draw, circle] {}}
        child {node[draw, circle] {}}
      }
      child {node[draw, circle] {}}
      child {node[draw, circle] {}
        child {node[draw, circle] {}}
        child {node[draw, circle] {}
          child {node[draw, circle] {}}
          child {node[draw, circle] {}
            child {node[draw, circle] {}}
          }
        }
        child {node[draw, circle] {}}
      };
    \end{tikzpicture}\hspace{1cm}
    \begin{tikzpicture}[scale=0.7, inner sep=2.5pt,
      level 2/.style={sibling distance=0.62cm}]
      \node[draw, circle] {}
      child {node[draw, circle] {}
        child {node[draw, circle] {}
          child {node[draw, circle] {}
            child {node[draw, circle] {}}
          }
          child {node[draw, circle] {}}
        }
        child {node[draw, circle] {}}
        child {node[draw, circle] {}}
      }
      child {node[draw, circle] {}
        child {node[draw, circle] {}}
        child {node[draw, circle] {}}
      }
      child {node[draw, circle] {}};
    \end{tikzpicture}
    \subcaption{Two embeddings of a given P\'olya tree}
  \end{subfigure}
  \begin{subfigure}[b]{0.35\linewidth}
    \centering
    \begin{tikzpicture}[scale=0.7, inner sep=2pt]
      \node[draw, circle] (1) at (90:2.5) {$1$};
      \node[draw, circle] (2) at (130:2.5) {$2$};
      \node[draw, circle] (3) at (170:2.5) {$3$};
      \node[draw, circle] (4) at (210:2.5) {$4$};
      \node[draw, circle] (5) at (250:2.5) {$5$};
      \node[draw, circle] (6) at (290:2.5) {$6$};
      \node[draw, circle] (7) at (330:2.5) {$7$};
      \node[draw, circle] (8) at (10:2.5) {$8$};
      \node[draw, circle] (9) at (50:2.5) {$9$};
      \draw (1) -- (5) -- (3) -- (2) (3) -- (4)
            (5) -- (9) -- (6) -- (7) (6) -- (8);
    \end{tikzpicture}
    \subcaption{A noncrossing tree of size $9$}
  \end{subfigure}
  \caption{P\'olya trees and noncrossing trees}
  \label{fig:more-tree-classes}
\end{figure}
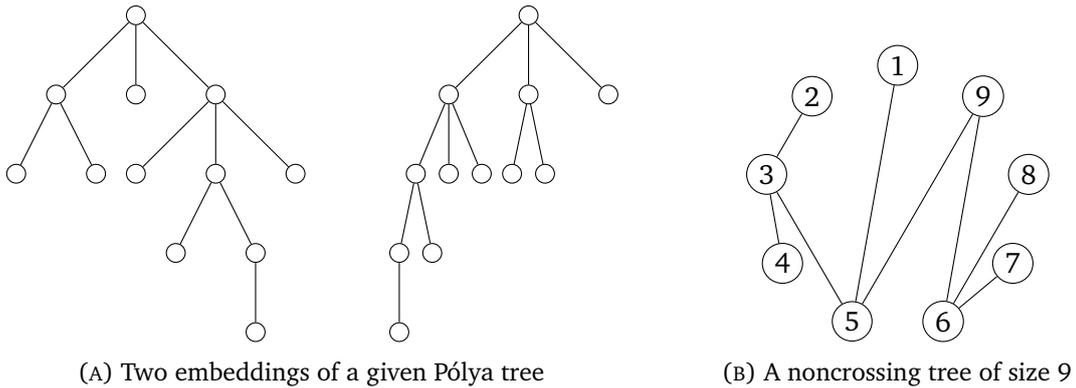

The basic principle in the analysis of both of these tree classes is the same: we leverage
the recursive nature of the respective family of trees to derive a functional equation for
$A_{r}(x,u)$. From there, similar techniques as in
Section~\ref{sec:simply-generated:analysis} (i.e., implicit differentiation and
propagation of the singular expansion of the basic generating function $F(x)$) can be used
to obtain (arbitrarily precise) asymptotic expansions for the mean and the variance of the
number of deleted nodes when cutting the tree $r$ times.

Qualitatively, in both of these cases we can prove a theorem of the following nature.

\begin{theorem}\label{thm:outlook:qualitative}
  Let $r\in\Z_{\geq 1}$ be fixed and $\mathcal{T}$ be either the family of P\'olya trees
  or the family of noncrossing trees. If $\tree_{n}$ denotes a uniformly random tree from
  $\mathcal{T}$ of size $n$, then for $n\to\infty$ the expected number of removed nodes
  when applying the ``cutting leaves'' procedure $r$ times to $\tree_{n}$ and the
  corresponding variance satisfy
  \[
    \E a_{r}(\tree_{n}) = \mu_{r} n + O(1),\quad \text{ and }\quad
    \V a_{r}(\tree_{n}) = \sigma_{r}^{2} n + O(1),
  \]
  for explicitly known constants $\mu_{r}$ and $\sigma_{r}^{2}$. Furthermore, the number
  of deleted nodes $a_{r}(\tree_{n})$ admits a Gaussian limit law.
\end{theorem}
Note that more precise asymptotic expansions for the mean and the variance (with
explicitly known constants) can also be computed.

\bibliography{bib/cheub}
\bibliographystyle{bibstyle/amsplainurl}

\end{document}